\documentclass[11pt,a4paper,reqno]{amsart}

\usepackage{amssymb}
\usepackage{setspace}
\usepackage{amsmath}
\usepackage{amsthm}
\usepackage[utf8]{inputenc}
\usepackage[numbers]{natbib}
\usepackage[english]{babel}
\usepackage[T1]{fontenc}
\usepackage{verbatim}
\usepackage{palatino}
\usepackage{dsfont}
\usepackage{upgreek}
\usepackage{amsfonts}
\usepackage{bbm}
\usepackage{mathrsfs}
\usepackage[dvipsnames]{xcolor}
\usepackage{mathtools}
\usepackage{tikz}
\colorlet{linkequation}{black}
\usepackage[colorlinks=true,  citecolor=black]{hyperref}
\hypersetup{linktocpage}
\allowdisplaybreaks

\newtheorem{theorem}{Theorem}[section]
\newtheorem{proposition}[theorem]{Proposition}

\newtheorem{lemma}[theorem]{Lemma}
\newtheorem{definition}[theorem]{Definition}

\newtheorem{remark}[theorem]{Remark}
\newtheorem{corollary}[theorem]{Corollary}

\newtheorem{fact}[theorem]{Fact}

\DeclareMathOperator\conv{conv}

\DeclareMathOperator\diam{diam}
\DeclareMathOperator\dist{dist}
\newcommand{\R}{\mathbb{R}}

\newcommand{\N}{\mathbb{N}}

\theoremstyle{remark}

\address{ 
Department of Mathematics\\
University of Connecticut\\
Storrs, Connecticut 06269
}
\email{zackary.boone@uconn.edu}

\address{
	Department of Mathematics\\
	Virginia Tech\\
	Blacksburg, Virginia 24060
	}
\email{palsson@vt.edu}

\keywords{Newhouse thickness, distance sets, point configurations}
\subjclass[2010]{Primary 28A75, Secondary 52C10}

\date{\today}

\thanks{The work of the second listed author was supported in part by Simons Foundation Grant \#360560.}

\begin{document}
\title{A pinned Mattila-Sj{\"o}lin type theorem for product sets}
\author{Zack Boone and Eyvindur Ari Palsson}
\maketitle

\begin{abstract}
We generalize a result of McDonald and Taylor which concerns the size of the tuples of edge lengths in the set $C_1 \times C_2$ utilizing the notion of thickness. Specifically, we show that $C_1, C_2 \subset \R^d$ compact sets with thickness satisfying $\tau(C_1) \tau(C_2) >1$, then the edge lengths in $C_1 \times C_2$ corresponding to any pinned finite tree configuration has non-empty interior. Originally proven for Cantor sets on the real line by McDonald and Taylor, we use the notion of thickness introduced by Falconer and Yavicoli which allows us to generalize the result of McDonald and Taylor to compact sets in $\R^d$.
\end{abstract}

\renewcommand{\contentsname}{\normalsize Contents}

\vspace*{-6mm}
  \tableofcontents

\section{Introduction}
Falconer's distance problem conjectures a relationship between the structure of a set $E$ and the size of its distance set $\Delta(C) := \{ |x-y| : x,y \in C\}$. The notion of structure and size used by Falconer is Hausdorff dimension and Lebesgue measure denoted $\dim_H(\cdot )$ and $\mathcal{L}(\cdot)$ respectively. In \cite{falconer_original}, Falconer proved that if $C \subset \R^d$ is compact with $\dim_H(C) > \frac{d+1}{2}$ then $\mathcal{L}(\Delta(C)) > 0$. The conjecture is that we can lower the threshold $\frac{d+1}{2}$ to $\frac{d}{2}$. Many improvements have been made on this threshold. When $d \geq 3$, the best known threshold is $\frac{d^2}{2d-1}$ which was first achieved by  Du, Guth, Ou, Wang, Wilson and Zhang in \cite{a_lot} when $d = 3$ and generalized to higher dimensions by Du
and Zhang in \cite{du_zhang}. Interestingly, this threshold can be improved for even dimensions. When restricting $d$ to even integers and for $d \geq 2$, the current threshold is $\frac{d}{2} +\frac{1}{4}$. For $d=2$, this threshold was obtained by Guth, Iosevich, Ou and Wang in \cite{stuff} and generalized to higher even dimensions by Du, Iosevich, Ou, Wang and Zhang in \cite{based}.

Another version of the Falconer distance problem, the pinned distance problem, has also seen much attention. Instead of looking at all possible distances from any two elements in a set, we fix one element in the set and only look at the distances the exist from this fixed point, i.e. for $x \in C$ the pinned distance set is $\Delta_{x}(C) := \{ |x - y| : y \in C\}$. We can then ask what dimensional requirements are needed so that $\mathcal{L}(\Delta_{x}(C)) > 0$. This is a stronger version of the Falconer distance problem since we can write $$\Delta(C) = \bigcup_{x \in C} \Delta_{x}(C)$$ so that if $\Delta_{x}(C)$ has positive Lebesgue measure for some $x \in C$, then we automatically get that $\Delta(C)$ has positive Lebesgue measure. This problem was first investigated by Peres and Schlag in \cite{peres_schlag} where they proved $$\dim_H \left(\left\{ x \in \R^d : \mathcal{L}(\Delta_x(E)) = 0\right\}  \right)\leq d+1 - \dim_H(E) $$ which, in particular, implies that if $\dim_H(E) > \frac{d+1}{2}$ then there exists $x \in E$ such that $\mathcal{L}(\Delta_x(E)) > 0$. Refinements on this problem have been made, particularly when restricted to the plane. For example, Shmerkin in \cite{shmerkin_improvement} proved that for $E \subset \R^2$ one has $$\dim_H\left( \{ x\in \R^2 : \dim_H(\Delta_x(E)) < 1\} \right) \leq 1 $$ when $E$ is a set satisfying $1 < \dim_H(E) = \dim_P(E)$, where $\dim_P(\cdot )$ denotes packing dimension. Moreover, this was improved by Keleti and Shmerkin in \cite{keleti_shmerkin} to the relaxed setting of when $\dim_P(E) \leq 2\dim_H(E) -1$.

Due to Mattila \cite{mattila_textbook_2015}, one method to prove the Falconer distance problem is to show that there exists a measure $\mu$ on $E$ such that $$\int_1^{\infty} \left( \int_{S^{d-1}} |\widehat{\mu}(r \omega) d\omega \right)^2 r^{d-1} dr < \infty, $$ and this approach has been called the Mattila integral approach. Surprisingly, the dimensional thresholds for the pinned and non-pinned versions of the distance problem match when one studies them via the Mattila integral approach. This is due to Liu in \cite{liu_identity} where he establishes an identity for $L^2$ integrals.

A variant of the distance problem asks to find dimensional conditions on $C$ such that not only will it be true that $\mathcal{L}(\Delta(C)) > 0$ but that $\Delta(C)$ will contain an interval. Let $(\cdot)^{\circ}$ denote the interior of a set. This problem was also investigated by Peres and Schlag in \cite{peres_schlag} where they proved a pinned version of this problem. In particular, they showed
\begin{equation}\label{eq:peres_schlag}
    \dim_H \left(\left\{ x \in \R^d : (\Delta_x(E))^{\circ} = \emptyset \right\}  \right) \leq d+2 - \dim_H(E)
\end{equation} which shows that if $\dim_H(E) > \frac{d+2}{2}$ then one can find $x\in E$ such that $\Delta_x(E)$ contains an interval. Iosevich and Liu in \cite{iosevich_liu} obtained the bound $\frac{(d+1)^2}{d-1} - \frac{d+1}{d-1}\dim_H(E)$ which improves on Equation \eqref{eq:peres_schlag} when $\dim_H(E) > \frac{d+3}{2}$.  In the non-pinned setting, Mattila and Sj{\"o}lin in \cite{mattila_sjolin} refined this result when they proved that if $\dim_H(C) > \frac{d+1}{2}$ then $\Delta(C)$ contains an interval, now known as the Mattila-Sj{\"o}lin theorem. Finding when some kind of geometric operation on a set has non-empty interior is an active research area. For example Iosevich, Mourgoglou, and Taylor in \cite{general_nonempty_interior} generalized the Mattila-Sj{\"o}lin theorem to the case where the set  $\Delta_B(E) = \{ \| x- y\|_B : x,y \in E\}$ has non-empty interior, where $\| \cdot \|_B$ is the metric induced by the norm defined by a symmetric bounded convex body $B$
with a smooth boundary and everywhere non-vanishing Gaussian curvature. Improvements on this threshold have been made by Koh, Pham, and Shen in \cite{product_sets} for product sets, i.e. sets of the form $C = E\times \cdots \times E= E^d \subset \R^d$ where $E \subset \R$ compact.

Another example of this, which serves as the basis of this paper, is Theorem 1.8 in \cite{mcdonald_taylor} proven by McDonald and Taylor which will be stated below, and it determines the amount of edge lengths one can find in the set $K_1 \times K_2$ when $K_1, K_2 \subset \R$ are Cantor sets.  Instead of taking Hausdorff dimension to be the structure of the set, McDonald and Taylor take thickness, denoted by $\tau$, to be the structure of their sets.  The definition of thickness is defined below.

 First introduced by Newhouse in \cite{newhouse}, he gave a definition of thickness specifically for Cantor sets on the real line. Newhouse used this notion of thickness to prove that if $K_1, K_2 \subset \R$ are Cantor sets such that $\tau(K_1) \tau(K_2) > 1$ then $K_1 \cap K_2 \neq \emptyset$, a result now known as the Newhouse Gap Lemma. Problems of determining elements in distance sets can be converted into problems of determining intersections of sets, so having a notion of structure which allows one to determine when two sets intersect can be of high importance to questions related to distance sets. Falconer and Yavicoli in \cite{falconeryavicoli} generalized both the definition of thickness and the Newhouse Gap Lemma to arbitrary compact sets in $\R^d$. Before stating the definition we describe the "gaps" of a compact set. \begin{definition}\label{def:gaps}
    Given a compact set $C \subset \R^d$, we define $\{G_n\}_{n=1}^{\infty}$ to be the, at most, countably many open bounded path-connected components of $C^C$ (the complement of $C$) and $E$ to be the unbounded open path-connected component of $C^C$ (except when $d=1$ when $E$ consists of two unbounded intervals). We call $E$ the \textbf{unbounded gap} of $C$ and $\{ G_n\}_{n=1}^{\infty}$ the \textbf{unbounded gaps} of $C$. Furthermore, we assume that the sequence of bounded gaps $\{G_{n}\}_{n=1}^{\infty}$ is ordered by non-increasing diameter. 
    \end{definition}
    We will write $\dist$ for the usual distance between points of non-empty subsets of $\R^d$ and $\diam$ for the diameter of a non-empty subset of $\R^d$.
    
    \begin{definition}[Thickness in $\R^d$]\label{def:thickness}
    The \textbf{thickness} of  $C$ is $$\tau(C) := \inf_{n \in \N} \frac{ \dist\left( G_n , \bigcup_{1 \leq i \leq n-1} G_i \cup E\right)}{\diam(G_n)}, $$ provided that $E$ is not the only path-connected component of $C^C$. When the only complementary path-connected component is $E$, we define \[
  \tau(C) =
  \begin{cases}
                                   +\infty & \text{if $C^{\circ} \neq \emptyset$} \\
  0 & \text{if $C^{\circ} = \emptyset$}
  \end{cases}
\]
We say $C$ is \textbf{thick} if $\tau(C) > 0$.
    \end{definition}
When restricted to $\R$, this definition of thickness is equivalent to one given by Newhouse. With this definition in mind, the kinds of sets that we are considering are ones that look like they have holes which are been punched out of them. Another way of putting it is that the sets we consider need to have some kind of border. Thus, one drawback to this definition is that there are sets which are of importance in geometric measure theory but do not fit into the framework of Definition \ref{def:thickness}. One example is the four corner Cantor set (and its generalizations) because it does not have any bounded gaps. Therefore it should be noted that recently, Yavicoli in \cite{yavicoli_new} gave a different definition of thickness and proved a type of gap lemma which did allow her to study sets such as the four corner Cantor set. Nevertheless, Falconer and Yavicoli in \cite{falconeryavicoli} used the definition of thickness given above to give a  generalized version of the Newhouse Gap Lemma which is stated as follows:
\begin{theorem}[Gap lemma in $\R^d$]\label{thm:gap_lemma}
    Let $C_1, C_2 \subset \R^d$ be compact such that neither of them is contained in a gap of the other and $\tau(C_1)\tau(C_2) > 1$. Then $C_1\cap C_2 \neq \emptyset$.
    \end{theorem}
 Before stating the result of McDonald and Taylor, we provide some preliminary definitions on graphs and specifically, trees.
 \begin{definition}[Graphs]\label{def:graphs}
A (finite) \textbf{graph} is a pair $G = (V, E)$ where $V$ is a (finite) set and $E$ is a set of 2-element subsets of $V$. If $\{ i, j\} \in E$ we say $i$ and $j$ are \textbf{adjacent} and write $i \sim j$.
\end{definition}

\begin{definition}[Chain and tree graphs]\label{def:tree graph}
The k-\textbf{chain} is the graph on the vertex set $\{1, \cdots, k+1\}$ with $i \sim j$ if and only if $|i - j| =1$. A \textbf{tree} is a connected, acyclic graph; equivalently, a tree is a graph in which any two vertices are connected by exactly one path. If $T$ is a tree, the \textbf{leaves} of $T$ are the vertices which are adjacent to exactly one other vertex of $T$.
\end{definition}
\begin{definition}[G distance sets]\label{def:G distance sets}
Let $G$ be a graph on the vertex set $\{1, \cdots , k+1\}$ with $m$ edges, and let $\sim$ denote the adjacency relation on $G$. Define the $G$ \textbf{distance set} of $B$ to be $$\Delta_{G}(B) := \left\{ \left( |x^i - x^j| \right)_{i \sim j} : x^1, \cdots, x^{k+1} \in B, x^i \neq x^j\right\},$$ where $(a_{i,j})_{i \sim j}$ denotes a vector in $\R^m$ with coordinates indexed by the edges of $G$.
\end{definition}
The usual definition of the distance set is a $1$-chain except with the element $0$ removed, i.e. if we make $G$ a $1$-chain then $\Delta(B) = \Delta_G(B) \cup \{ 0\}$. The element $0$ is removed from $G$ distance sets because we require non-degeneracy conditions in some of the later proofs. McDonald and Taylor in \cite{mcdonald_taylor} used Newhouse's definition to prove the following theorem:
 \begin{theorem}[\textit{Theorem 1.8 in \cite{mcdonald_taylor}}]\label{thm:reference_theorem}
     Let $K_1, K_2 \subset \R$ be Cantor sets satisfying $\tau(K_1) \tau(K_2) > 1$. For any finite tree $T$, the set $\Delta_T(K_1 \times K_2)$ has non-empty interior.
 \end{theorem}
 The goal of this paper is to extend the above theorem to compact sets in $\R^d$ satisfying the same thickness condition, and also extending the work to a pinned tree distance set.
 This theorem uses and extends the work from Simon and Taylor in \cite{simon_taylor}. In that paper, they also prove a Mattila-Sj{\"o}lin type theorem by considering thickness instead of Hausdorff dimension. In particular, Corollary 2.12 in \cite{simon_taylor} states that for $K_1 , K_2 \subset \R$ Cantor, if $\tau(K_1)\tau(K_2) > 1$ then $$\left( \Delta_{x}(K_1 \times K_2) \right)^{\circ} \neq \emptyset.$$ Extensions to pinned tree distance sets for Cantor sets in $\R$ can also be concluded from the work of McDonald and Taylor, which is discussed below.
 
 A key ingredient in proving statements about non-empty interior of tree distance sets is having access to a strong statement about pinned distance sets, which is presented next and is adapted and extended from Lemma 3.5 in \cite{mcdonald_taylor}.
\begin{theorem}\label{thm:general}
	Let $C_1, C_2 \subset \R^d$ be compacts sets satisfying $\tau(C_1) \tau(C_2) >1$. Then for any $x^0 \in C_1 \times C_2$ there exists an open neighborhood $S$ about $x^0$ such that  $$ \bigcap_{x \in S}\Delta_{x} (C_1 \times C_2)$$ has non-empty interior.
	\end{theorem}
The proof of this theorem is given in Section \ref{section:main}. We note that this result not only states existence of such an $x^0$, but rather we can take any $x^0 \in C_1 \times C_2$ and obtain this result. Furthermore, this theorem gives a Mattila-Sj{\"o}lin type result in the pinned distance setting. Indeed, since we are taking a neighborhood around $x^0$, we immediately obtain the following corollary:
\begin{corollary}
    Let $C_1, C_2 \subset \R^d$ be compacts sets satisfying $\tau(C_1) \tau(C_2) >1$. Then for any $x \in C_1 \times C_2$ we have $$\left( \Delta_x(C_1 \times C_2) \right)^{\circ} \neq \emptyset. $$
\end{corollary}
In Section \ref{section:application} we construct sets $C_1 \times C_2$, where $C_1, C_2 \subset \R^d$, for which the above theorem can be applied and such that the dimension of $C_1 \times C_2$ approaches $2d-1$. Then note that when $d=1$ (so $C_1 \times C_2 \subset \R^2)$, the threshold of $2d-1$ turns into $1$, which matches the threshold in the Falconer distance conjecture and both the pinned Falconer problem as well as the Mattila-Sj{\"o}lin variant. This gives some evidence to the conjecture that we should expect the same threshold for the pinned Mattila-Sj{\"o}lin question. We also point out that, since we are looking at the product of sets, the threshold that comes from Equation \eqref{eq:peres_schlag} is $\frac{2d+2}{2}$. Thus, we can find a sequence of sets with dimension converging to $2d-1$, which falls below $\frac{2d+2}{2}$ for $d=1$ and matches it when $d=2$. However, for $d\geq 3$, the threshold $\frac{2d+2}{2}$ is lower than $2d-1$ (and thus, better).
 
 Before stating the main result, we fix some notation. Let $T^k$ denote a tree with $k$ edges, i.e. a tree with $k+1$ vertices. Then we define the pinned tree distance set with $k$ edges by $$\Delta_{T_x^k}(C) := \{ (|x^i - x^j|)_{i\sim j} : x^1, \cdots , x^{k+1} \in C, x^i \neq x^j, x = x^i \text{ for some } i \}.$$
Using the above theorem, we can extend the work of McDonald and Taylor to the following, which is the main result of this paper:
\begin{theorem}\label{thm:pinned_tree_set}
Let $C_1, C_2 \subset \R^d$ be compact sets such that $\tau(C_1)\tau(C_2) >1$. Then for any finite pinned tree $T_x^k$ of $k+1$ vertices, the set $\Delta_{T_x^k}(C_1 \times C_2)$ has non-empty interior.
\end{theorem}
Note that the requirements for Theorem \ref{thm:pinned_tree_set} are no more stringent than the requirements for Theorem \ref{thm:general}. Thus, we again can construct a wide range of sets of the form $C_1 \times C_2 \in \R^{2d}$ of with dimension going down to dimension $2d-1$.

To conclude the introduction, we note that there have been numerous other results on establishing the existence of configurations of a set from the perspective of Hausdorff dimension. For example, Palsson and Romero Acosta in \cite{francisco_triangles} establish that when $\dim_H(E) > \frac{2}{3}d +1$ for $d \geq 4$, then the set of congruence classes of triangles formed by triples of points in $E$ will have non-empty interior, with further improvements in low dimensions in a very recent preprint \cite{francisco_simplices_low_dim}.

As another variant of the problem, we can describe distances corresponding to a tree as a function $\Phi$, which is a vector-valued function describing the distances between a set of points. There has also been work when allowing $\Phi$ to describe more general operations, and determining when the resulting set under $\Phi$ has non-empty interior. This was done by Greenleaf, Iosevich, and Taylor in \cite{general_phi}
where they considered a suitable class of functions $\Phi : \R^d \times \R^d \rightarrow \R^k$ to establish the resulting 2-point configuration set $\Delta_{\Phi}(E) := \{\Phi(x,y) : x,y \in E \}$ has non-empty interior given lower bounds on $\dim_H(E)$ which depend on $\Phi$. In \cite{k-point_configurations}, the same authors built upon this work to establish similar results when considering more general k-point configurations and in their most recent paper \cite{microlocal_partition} they can even handle triangles matching the result in \cite{francisco_triangles}.

The proof of Theorem \ref{thm:pinned_tree_set} will be done by induction on $k$ (the number of edges in the tree), and heavily relying on Theorem \ref{thm:general}. Proving Theorem \ref{thm:general} will largely follow the same strategy as is done in \cite{mcdonald_taylor}, however there are clear difficulties that arise when going to higher dimensions, and these issues will be discussed throughout this document

\vskip2em
\section{Application to Sierpi{\'n}ski carpets}\label{section:application} Here, we give an example of a sequence of sets, which will be Sierpi{\'n}ski carpets, for which we can apply Corollary \ref{thm:pinned_tree_set} and for which the dimension of the sets in this sequence converges to $2d-1$. For a more detailed discussion and construction of these sets see \cite{falconeryavicoli}. For simplicity, we will assume each function is the iterated function system of the following Sierpi{\'n}ski carpets will have the same contraction ratios. That is, let $n \in \N_{\geq 3}$ be an odd number and set $$D = \{\bold{i} = (i_1, \cdots , i_d) : 1 \leq i_k \leq n \text{ with } (i_1, \cdots , i_d) \neq ((n+1)/2, \cdots, (n+1)/2)   \} .$$ Then the family of maps $\{ f_{\bold{i}} : \bold{i} \in D \}$, where $f_{\bold{i}} :\R^d \rightarrow \R^d$ and is defined by $$f_{\bold{i}}(x_1 , \cdots , x_d) = \left(\frac{x_1+i_1-1}{n}, \cdots , \frac{x_d+i_d-1}{n}  \right), $$ forms an iterated function system, so that there exists a unique and non-empty compact $C \subset \R^d$ such that $C = \bigcup_{\bold{i} \in D} f_{\bold{i}}(C)$. Furthermore, $\dim_H(C) = s$ where $s$ is the unique number satisfying $\#D \left(\frac{1}{n}\right)^s = 1$. For verification of these facts, see \cite{hutchinson}. This means $\dim_H(C) = \frac{\log(\# D)}{\log(n)}$.  The set $C$ corresponds to initially breaking up the unit cube $[0,1]^d$ into $n^d$ boxes and removing the middle one, then iteratively performing this process on the remaining boxes. Since we are only removing the middle one, we have $\#D = n^d-1$. Thus, $\dim_H(C) = \frac{\log(n^d-1)}{\log(n)}$. Furthermore, in \cite{falconeryavicoli} it is shown that $\tau(C) = \frac{n-1}{2\sqrt{d}}.$

Another kind of Sierpi{\'n}ski carpet that we can calculate the Haudsorff dimension and thickness of is one in which instead of just removing the middle box, we only keep the "outer boundary" of boxes. That is, for $n \geq 5$ and $n$ odd, we remove the middle $(n-2)^d$ squares and keep the boxes on the outer edge, then perform this process iteratively on the remaining squares. For this set, say $C$, we have $\dim_H(C) = \frac{\log(n^d-(n-2)^d)}{\log(n)}$. By a similar discussion as found in \cite{falconeryavicoli}, the thickness is then $$\tau(C) = \frac{\frac{1}{n}}{\sqrt{d\left(\frac{n-2}{n} \right)^2}} = \frac{1}{(n-2)\sqrt{d}}.$$ Thus, if we take $C_n$ to be a Sierpi{\'n}ski carpet of the first kind described, with contraction ratio $1/n$, and $C_m$ to be a Sierpi{\'n}ski carpet of the second kind described, with contraction ratio $1/m$. Then if $1 < \tau(C_n) \tau(C_m)$, this yields $2d < \frac{n-1}{m-2}$. If $n$ is fixed, this inequality is satisfied, for example, by taking $m$ to be largest odd integer such that $m < \frac{n-1 + 4d}{2d}$. So an $n \rightarrow \infty$, we choose $m_n$ to a sequence of the largest odd numbers such that $m_n < \frac{n-1 +4d}{2d}$. Then $\dim_H(C_n) \rightarrow d$ and $\dim_H(C_{m_n}) \rightarrow d-1$ as $n\rightarrow \infty$. Thus, for any $x \in C_n \times C_{m_n}$ and any $k \in \N$, the set $\Delta_{T_x^k}(C_n \times C_{m_n})$ has non-empty interior and $$ \dim_H(C_n \times C_{m_n}) = \dim_H(C_n) + \dim_H(C_{m_n}) \rightarrow 2d-1 $$ as $n \rightarrow \infty$. 
\begin{comment}
As stated earlier in this section, the main tool for attacking the distance and pinned distance problem is via the Mattila integral approach. There are other approaches though, one example being the use of Fourier integral operators. In \cite{iosevich_liu}, Iosevich and Liu establishes dimensional thresholds for the pinned distance problem via this approach, obtaining a threshold of $\frac{(d+1)^2}{2d}$. Thus, we quickly note that we can obtain $C_n , C_{m_n} \subset \R^2$ with dimension nearly $2d-1 = 3$ while the dimensional threshold $\frac{(2d+1)^2}{4d}$ ($2d$ since we are considering $C_n \times C_{m_n}$) turns to $\frac{25}{8} = 3.125$, so a slightly better threshold can be obtained.
\end{comment}

\begin{comment}
The proof strategy of Theorem \ref{thm:pinned_tree_set} will be largely the same as is done in \cite{mcdonald_taylor}, however there are clear difficulties that arise when going to higher dimensions, and these issues will be discussed throughout this document. The first step in the process is to convert the problem of finding non-empty interior of tree structures into a suitable form, which is done in Section \ref{section:tree}.
\end{comment}

\section{Limitations of the Techniques}
First, note that Theorem \ref{thm:pinned_tree_set} immediately establishes the following result:
\begin{corollary}\label{thm:main_3}
    Let $C_1, C_2 \subset \R^d$ be compacts sets such that $\tau(C_1) \tau(C_2) >1$. Then for any finite tree $T$, the set $\Delta_T(C_1 \times C_2)$ has non-empty interior.
\end{corollary}
One can then ask how low can we go, in terms of Hausdorff dimension, to obtain the result of the above theorem? Falconer and Yavicoli in \cite{falconeryavicoli} provide two estimates on this subject. The first of these estimates uses the following definitions: \begin{equation*}
    K_1 := \frac{2d (24 \sqrt{d})^d \log \left( 16\sqrt{d} \right)}{1- \frac{1}{2^d}}\,\,\,\,\,\,\,\,\,\,\,\,\,\,\,\, and \,\,\,\,\,\,\,\,\,\,\,\,\,\,\, K_2 := \left( \frac{(24 \sqrt{d})^d (1+4^d 2)}{1- \frac{1}{2^d}} \right)^2
\end{equation*}
where $d$ is being taken as in $\R^d$. 
\begin{theorem}[\textit{Corollary 19 in \cite{falconeryavicoli}}]\label{thm:first_lower_bound_estimate}
    Let $C$ be a compact set in $\R^d$ with positive thickness $\tau$ (so $\diam(C) < + \infty$ and there is a ball $B$ such that $B \cap E \neq \emptyset$ where $E$ is the unbounded gap of $C$). If there exists $c \in (0, d)$ such that $\tau(C)^{-c} \leq \frac{1}{K_2}\beta^c (1- \beta^{d-c})$ for $\beta := \min \{ \frac{1}{4}, \frac{\diam(B)}{\diam(C)} \}$ then $$ \dim_H(C) \geq \dim_H(B \cap C) \geq d- K_1 \frac{\tau(C)^{-d}}{\beta^d |\log (\beta) |} > 0. $$
\end{theorem}
\begin{theorem}[\textit{Proposition 21 in \cite{falconeryavicoli}}]\label{thm:second_lower_bound_estimate}
    Let $C_0 \subset \R^d$ be a proper compact convex set, and let $C = C_0 \backslash \bigcup_{i=1}^{\infty} G_i$ where $\{G_i\}_i$ are open convex gaps. Then $$\dim_H(C) \geq d-1 + \frac{\log 2}{\log (2+ 1/ \tau(C))}. $$
\end{theorem}
For $C \subset \R^d$, since $\dim_H(C \times C) \geq 2 \dim_H(C)$ the lower estimates on $\dim_H(C\times C)$ will be of the form $2\left(d - K_1 \frac{\tau(C)^{-d}}{\beta^d |\log (\beta) |} \right)$ and $2 \left(d-1 + \frac{\log 2}{\log (1+ 1/\tau(C))} \right)$. For large thickness, the first estimate is lower (and therefore better since we want to know how small we can make the dimension) and for small thickness, the second estimate is lower (and therefore better). But when considering compact sets $C$ such $\tau(C)$ is close to $1$, this restricts us to the second estimate. This is because to fulfill the requirement $\tau(C)^{-c} \leq \frac{1}{K_2}\beta^c (1- \beta^{d-c}) \leq \frac{1}{K_2}$, since $K_2$ is very large, we need to $\tau(C)$ to be much larger than $1$. Thus, for $\tau(C) >1$ and observing the second estimate, we have $$\dim_H(C \times C) \geq 2 \left(d-1 + \frac{\log 2}{\log (2 + 1/\tau(C))} \right) > 2\left(d-1 + \frac{\log 2}{\log 3} \right).$$ So for $\tau(C)$ close to $1$, we have the Mattila-Sj{\"o}lin type result of Theorem \ref{thm:main_3} where we allow the dimension of our set to get down to $2 \left( d-1 + \frac{\log 2}{\log 3} \right)$. 
Note that this lower bound is worse than what is found in the Mattila-Sj{\"o}lin theorem, which in this case would be $\frac{2d+1}{2}$, with the only improvement on the bound being the case when $d=1$. However, there has been work done on establishing chain and tree configurations from the point of view of Hausdorff dimension rather than thickness. Bennett, Iosevich, and Taylor in \cite{recent_chain} were able to establish the existence of finite chain configurations with the Mattila-Sj{\"o}lin threshold of $\frac{d+1}{2}$, i.e. the assumption they were placing on sets was that $E \subset \R^d$ for $d \geq 2$ and $\dim_H(E) > \frac{d+1}{2}$, which is enough to guarantee the existence of chain configurations. Building on this work, Iosevich and Taylor in \cite{recent_tree} generalized the result to establish the existence of tree configurations using the  Mattila-Sj{\"o}lin threshold.

\section{Discussion of the distance function}\label{section:distance_function}
We now introduce our distance function, $g_{x,t} : \R^d \rightarrow \R^d$. For a fixed $x \in \R^d$ this function will have the property that for $C_1, C_2\subset \R^d$ we will have $t \in \Delta_x (C_1 \times C_2)$ provided that $C_2 \cap g_{x,t}(C_1) \neq \emptyset$. As pointed out in the introduction, problems of distance sets can be turned into problems of finding when sets intersect. This is one variant of that idea. Now fix $(x,t) \in \R^{2d} \times [0, \infty)$ and for $i \in \{1, \cdots , d\}$ define $\widetilde{g}^i : \R^d \rightarrow \R$ by \begin{align*} \widetilde{g}^i(y_1, \cdots ,y_d) &:= x_{i+d} + \sqrt{
	\begin{aligned}
	\frac{1}{d}t^2 - &0 \cdot (x_1 - y_1)^2 - \cdots - 0 \cdot (x_{i-1} - y_{i-1})^2 \\
	&- (x_i - y_i)^2 - 0 \cdot (x_{i+1} - y_{i+1})^2 - \cdots - 0  \cdot (x_d - y_d)^2
	\end{aligned}
	}
	\\ &= x_{i+d} + \sqrt{\frac{1}{d}t^2 - (x_i - y_i)^2} =: g^i(y_i).
	\end{align*}
	We do this because if we set $y_{i+d} = g^i(y_i)$ then $\frac{1}{d}t^2 = (y_{i+d} - x_{i+d})^2 + (y_i - x_i)^2$ for all $i$, which then means that by summing each of these terms we will have $$\sum_{i=1}^{2d}(y_i - x_i)^2 = t^2$$ or in other words, the distance between $(y_1, \cdots ,y_{2d})$ and $(x_1 , \cdots , x_{2d})$ is $t$. Finally defined $g_{x,t} : \R^d \rightarrow \R^d$ by $$g_{x,t}(y_1, \cdots , y_d) = (\widetilde{g}^1(y_1, \cdots , y_d) , \cdots , \widetilde{g}^d(y_1, \cdots , y_d)) = (g^1(y_1), \cdots , g^d(y_d)).$$ Then we do indeed obtain the property that $t \in \Delta_x(C_1 \times C_2)$ provided $C_2 \cap g_{x,t}(C_1) \neq \emptyset$.
	
	Now we list an interesting property of $g_{x,t}$. From here on out, $g_{x,t}'$ will denote the Jacobian matrix of $g_{x,t}$. Then for $z = (z_1, \cdots , z_d) \in \R^d$ we have 
$$
g_{x,t}'(z) = 
\begin{bmatrix}
\frac{d g^1}{dz_1} & 0 & 0 & \cdots & 0 \\
0 & \frac{d g^2}{dz_2} & 0 & \cdots & 0 \\ \vdots & & \ddots & & \vdots \\ 0 & 0 & \cdots & 0 & \frac{d g^d}{dz_d}
\end{bmatrix} =
\begin{bmatrix}
    \frac{x_1 - z_1}{\sqrt{\frac{1}{d}t^2 - (z_1 - x_1)^2}} & 0 & 0 & \cdots & 0 \\ 0 & \frac{x_2 - z_2}{\sqrt{\frac{1}{d}t^2 - (z_2 - x_2)^2}} & 0 & \cdots & 0 \\ \vdots & & \ddots & & \vdots \\
    0 & 0 & \cdots & 0 & \frac{x_d - z_d}{\sqrt{\frac{1}{d}t^2 - (z_d - x_d)^2}}
\end{bmatrix}.$$

Thus, $g'_{x,t}(z)$ is a diagonal matrix. Furthermore, the norm that we will be working with for $g'_{x,t}(z)$ is the operator norm which is the square root of the largest eigenvalue of $g'_{x,t}(z)^T g'_{x,t}(z)$ where $T$ denotes the transpose. Then because $g'_{x,t}(z)$ is diagonal, we have $g'_{x,t}(z)^T = g'_{x,t}(z)$ and therefore $$g'_{x,t}(z)^T g'_{x,t}(z) = \begin{bmatrix}
    \frac{(x_1-z_1)^2}{\frac{1}{d}t^2 - (z_1-x_1)^2} & 0 & 0 & \cdots & 0 \\ 0 & \frac{(x_2-z_2)^2}{\frac{1}{d}t^2 - (z_2-x_2)^2} & 0 & \cdots & 0 \\ \vdots & & \ddots & & \vdots \\ 0 & 0 & \cdots & 0 & \frac{(x_d-z_d)^2}{\frac{1}{d}t^2 - (z_d-x_d)^2}
\end{bmatrix} $$ 
so that \begin{equation}\label{eq:op_norm}\| g'_{x,t}(z) \|_{\text{op}} = \max \left\{ \frac{|x_1 - z_1|}{\sqrt{\frac{1}{d}t^2 - (z_1 - x_1)^2}}, \cdots , \frac{|x_d - z_d|}{\sqrt{\frac{1}{d}t^2 - (z_d - x_d)^2}} \right\} = \max \left\{ \left| \frac{d g^1}{dz_1}\right| , \cdots , \left| \frac{d g^d}{dz_d} \right| \right\}.
\end{equation}
This estimate will come in later. The function $g_{x,t}$ satisfies another useful property.
\vspace{1mm}

\begin{lemma}[Mean Value Theorem for $g_{x,t}$]\label{thm:MVT} Let $a,b \in \R^d$ be given such that $a_i < b_i$ for all $i \in \{1, \cdots , d\}$ and $g_{x,t}$ is defined on  $[a_1, b_1] \times \cdots \times [a_d, b_d]$. Then there exists $z_i \in (a_i, b_i)$ such that $$(b-a)g_{x,t}^{'}(z) = g_{x,t}(b) - g_{x,t}(a)$$ where $z = (z_1, \cdots, z_d)$.
\end{lemma} \begin{proof}
Since each $g^i$ is a mapping from $\R$ to $\R$ and will be differentiable in $[a_i , b_i]$, by the one-dimensional mean value theorem, there exists $z_i \in (a_i, b_i)$ such that $(b_i - a_i) \frac{dg^i}{dz_i} =g^i(b_i) - g^i(a_i).$ So 
\begin{align*}
(b-a)g_{x,t}^{'}(z) = 
\begin{bmatrix}
b_1 - a_1 \\
\vdots \\
b_d - a_d
\end{bmatrix}
\begin{bmatrix}
\frac{dg^1}{dz_1} & 0 & \cdots & 0 \\ \vdots & \ddots & &  \vdots \\
0 & 0 & \cdots &  \frac{dg^d}{dz_d}
\end{bmatrix}  = 
\begin{bmatrix}
(b_1-a_1)\frac{dg^1}{dz_1} \\ \vdots \\
(b_d-a_d)\frac{dg^2}{dz_d}
\end{bmatrix}  &= 
\begin{bmatrix}
g^1(b_1) - g^1(a_1) \\ \vdots \\
g^d(b_d) - g^d(a_d)
\end{bmatrix} \\ &= 
\begin{bmatrix}
g^1(b_1) \\ \vdots \\
g^d(b_d)
\end{bmatrix} -
\begin{bmatrix}
g^1(a_1) \\ \vdots \\
g^d(a_d)
\end{bmatrix} \\ &= 
g_{x,t}(b) - g_{x,t}(a).
\end{align*}
\end{proof}

\section{Proof of Theorem \ref{thm:pinned_tree_set}}\label{section:main}
We start this section with a lemma that allows us to write $\tau$ in a, sometimes, more convenient way.

\begin{lemma}\label{thm:another_thickness}

    Let $C \subset \R^d$ be compact with bounded gaps $\{G_n\}_{n=1}^{\infty}$. Set $\Lambda_n := \{ i \neq n: \diam(G_n) \leq \diam(G_i)\}$. Then we can also write $\tau$ as $$\tau(C) = \inf_{n \in \N} \frac{ \dist\left( G_n , \bigcup_{i \in \Lambda_n} G_i \cup E\right)}{\diam(G_n)}.  $$
    \end{lemma}
    \begin{proof}
    To see this, assume that we cannot. Set $$\widehat{\tau}(C) := \inf_{n \in \N} \frac{ \dist\left( G_n , \bigcup_{i \in \Lambda_n} G_i \cup E\right)}{\diam(G_n)}.$$ First note that since $\diam(G_n) \geq \diam(G_{n+1})$ we have $\widehat{\tau}(C) \leq \tau(C)$ as we could only be increasing the number gaps in our analysis using $\Lambda_n$ (as would be the case if we had $\diam(G_i) = \diam(G_n)$ with $i > n$). So therefore we may assume for contradiction that $\widehat{\tau}(C) < \tau(C)$. Thus, there exists $\eta_1 > 0$ such that $\widehat{\tau}(C) + \eta_1 = \tau(C)$. Since $\widehat{\tau}$ is an infimum, by properties of infimums we can find $N \in \N$ such that $$\frac{\dist \left( G_N , \bigcup_{i \in \Lambda_N} G_i \cup E \right)}{\diam(G_N)} < \widehat{\tau}(C) + \eta_1 = \tau(C).$$ Also note that this means $$\frac{\dist \left( G_N , \bigcup_{i \in \Lambda_N} G_i \cup E \right)}{\diam(G_N)} < \tau(C) \leq \frac{\dist \left( G_N , \bigcup_{1 \leq i \leq N-1} G_i \cup E \right)}{\diam(G_N)}$$ and therefore, the gap with the smallest distance is not achieved by $E$ since $E$ is present in both the LHS and RHS of the above inequality. Now set $\eta_2 > 0$ such that $$\frac{\dist \left( G_N , \bigcup_{i \in \Lambda_N} G_i \cup E \right)}{\diam(G_N)} + \eta_2 = \tau(C).$$ Because $\dist$ is an infimum property, because the minimal distance to $G_N$ is not achieved by $G_i$ with $1 \leq i \leq N-1$, and because we are always assuming $\diam(G_n) \geq \diam(G_{n+1})$, we can find $M \in \Lambda_N$ such that $M >N$, $\diam(G_M) = \diam(G_N)$, and $$\frac{\dist(G_N, G_M)}{\diam(G_M)} < \frac{\dist \left( G_N , \bigcup_{i \in \Lambda_N} G_i \cup E \right)}{\diam(G_N)} + \eta_2 = \tau(C).$$ But since $M > N$ and $\diam(G_M) = \diam(G_N)$ this implies $$\frac{\dist (G_N, G_M)}{\diam(G_N)} = \frac{\dist(G_M, G_N)}{\diam(G_M)} \geq \frac{\dist \left(G_M , \bigcup_{1 \leq i \leq M-1} G_i \cup E \right)}{\diam(G_M)}$$ which means $$\frac{\dist \left(G_M , \bigcup_{1 \leq i \leq M-1} G_i \cup E \right)}{\diam(G_M)} < \tau(C).$$ However this is a contradiction since $\tau$ is an infimum over all such elements and thus, we can write $$\tau(C) = \inf_n \frac{\dist \left( G_n , \bigcup_{1 \leq i \leq n-1} G_i \cup E \right)}{\diam (G_n)} = \inf_n \frac{\dist \left( G_n , \bigcup_{i \in \Lambda_n} G_i \cup E \right)}{\diam (G_n)}.$$
    \end{proof}
    
Throughout the proof of Theorem \ref{thm:general} and Lemma \ref{thm:thickness_nearly_preserved_dimension_d}, it will be convenient for our pinned point to be the origin and for us to be able work with boundary points of gaps of $C_1$ and $C_2$ to exist on the diagonal. This seems to be, initially, a highly restrictive condition to put on our sets $C_1$ and $C_2$. Fortunately though, we can do this by utilizing the fact that affine transformations are isometries and that thickness is preserved under affine transformations. The next lemma characterizes this shift.

\begin{lemma}\label{thm:nice_conditions}
Let $C_1, C_2 \subset \R^d$ with $\tau(C_1)\tau(C_2) > 1$. Let $x^0 \in C_1 \times C_2$ be given and let $u^1 \in C_1$, $u^2 \in C_2$ be such that they are boundary points of some gap (could also be the unbounded gaps) of $C_1$ and $C_2$ respectively. Then there exists affine transformations $S_1 : \R^d \rightarrow \R^d$ and $S_2 : \R^d \rightarrow \R^d$ with the following properties:
\begin{enumerate}
    \item $x^0 \mapsto \vec{0}$  under $S_1$ and $S_2$
    \item $u^1 \mapsto w^1 =(w_1^1 , \cdots, w_d^1)$ under $S_1$ such that $w_i^1 = w_1^1$ for all $i$ and $u^2 \mapsto w^2 = (w_1^2 , \cdots, w_d^2)$ under $S_2$ such that $w_i^2 = w_1^2$ for all $i$
    \item $\tau(S_1(C_1)) = \tau(C_1)$ and $\tau(S_2(C_2)) = \tau(C_2)$
    \item $\Delta_{\vec{0}} (S_1(C_1) \times S_2(C_2)) = \Delta_{x^0} (C_1 \times C_2)$.
\end{enumerate}
\end{lemma}
\begin{proof}
Let $x^{0,1} : = (x_1^0, \cdots , x_d^0)$ and $x^{0,2} : = (x_{d+1}^0 , \cdots , x_{2d}^0)$. Find $z^1, z^2 \in \R^d$ such that $x^{0,1}+z^1 = \vec{0}$ and $x^{0,2} + z^2 = \vec{0}$. Let $SO(d)$ denote the group of $d$-dimensional rotations. Then there exists $A_1 , A_2 \in SO(d)$ such that $A_1(u^1 + z^1) = w^1 = (w_1^1, \cdots , w_d^1)$ such that $w_i^1 = w_1^1$ for all $i \in \{1, \cdots, d\}$ and $A_2 (u^2 + z^2) = w^2 =(w_1^2, \cdots w_d^2)$ such that $w_i^2 = w_1^2$ for all $i$, i.e. we are rotating the points $u^i + z^i$ so that they are on the diagonal. Furthermore, elements in $SO(d)$ are linear. So for any $x \in \R^d$ we have that the function $S_i : \R^d \rightarrow \R^d$ defined by $$S_i(x) = A_i(x+ z^i) = A_i x + A_i z^i$$ which shows that $S_i$ is an affine transformation. Therefore, since thickness is preserved under affine transformations we get $\tau(S_i(C_i)) = \tau(C_i)$. Furthermore, both translations by $z^i$ and rotation by $A_i$ are isometries which means that $S_i$ is an isometry. Thus, for $t \in \Delta_{\vec{0}} (S_1(C_1) \times S_2(C_2))$ we can find $y^1 \in C_1$ and $y^2 \in C_2$ such that \begin{align*}
    t^2 &= \left[ \dist(S_1(x^{0,1}), S_1(y^1)) \right]^2 + \left[ \dist(S_2(x^{0,2}), S_2(y^2)) \right]^2 \\ &= \left[ \dist(x^{0,1} , y^1) \right]^2 + \left[ \dist(x^{0,2}, y^2) \right]^2 \\ &= \sum_{i=1}^d (x_i^0 - y_i^1)^2 + \sum_{j = d+1}^{2d} (x_j^0 - y_j^2)^2
\end{align*}
and therefore, $t \in \Delta_{x^0} (C_1 \times C_2).$ Since equality held throughout, we can say $\Delta_{x^0} (C_1 \times C_2) = \Delta_{\vec{0}}(S_1(C_1) \times S_2(C_2))$.
\end{proof}

\begin{comment}Furthermore, we prove Theorem \ref{thm:main_3} in an iterative way that utilizes the conditions of the above lemma. We start out with our initial tree structure as done in Theorem \ref{thm:tree_building_mechanism} and start at a leaf. We then rotate and translate in the manner as stated above, get our desired conclusion, then rotate and translate back to the original setting. Then we repeat the process of finding another leaf, rotate and translate, get desired conclusion, then rotate and translate back to the original setting.
\end{comment}

Since we are now working in $\R^d$ we now to deal with more complicated sets and in particular, more complicated boundaries. But the next series of facts, which is presented without proof, allows us nicely characterize how the boundary and diameter of a set is affected by $g_{x,t}$.
\begin{fact}
    Suppose $f$ is continuous and $H$ is open where $f(H) = H'$ are open sets with $H$ having compact closure. Then $f(\partial H) = \partial H'$.
\end{fact}
In our language, this will mean $g_{x,t}(\partial H) = \partial g_{x,t}(H)$ where $H$ is a gap of $C$.

\begin{fact}
    If $H$ is open, then $\diam(H) = \diam(\partial H)$.
\end{fact}

Combining these two facts gives us the next fact. 

\begin{fact}\label{thm:useful_fact}
Since $g_{x,t}$ is continuous we have $$\diam (g_{x,t}(H)) = \diam (\partial g_{x,t}(H)) = \diam (g_{x,t}(\partial H)).$$
\end{fact}
Thus, when trying to characterize the diameter of some gap under $g_{x,t}$, it suffices to look how the boundary is affected by $g_{x,t}$. Then next part of this discussion characterizes how thickness is changed under $g_{x,t}$.

As pointed out in Section \ref{section:distance_function}, we will have that $t \in \Delta_x(C_1 \times C_2)$ provided $C_2 \cap g_{x,t}(C_1) \neq \emptyset$. So to use the Gap Lemma we need to understand the thickness of the set $g_{x,t}(C_1)$. To do this, we introduce a notion of $\varepsilon$-thickness that gives us some breathing room to work with in later calculations. The following definition is adapted from Definition 3.2 in \cite{mcdonald_taylor}.

\begin{definition}\label{def:ep_thickness}
Let $C \in \R^d$ be a compact set with bounded gaps $\{ G_n \}_{n=1}^{\infty}$. Let $u_n \in \partial G_n$ and define $H_{\varepsilon ,n} := \{i \neq n: (1-\varepsilon)\diam(G_n) < \diam(G_i) \}. $ Then the $\varepsilon$-thickness of $C$ at $u_n$ is defined to be $$\tau_{\varepsilon}(C, u_n) :=  \frac{ \dist \left(u_n , \bigcup_{i \in H_{\varepsilon, n}} G_i \cup E\right)}{\diam(G_n)}$$ and the \textbf{$\varepsilon$-thickness} of $C$ is $$\tau_{\varepsilon}(C) = \inf_{n \in \N} \inf_{u_n} \tau_{\varepsilon}(C, u_n) $$ the infimum being taken over all boundary points of all gaps.
\end{definition}

This definition leads to two important properties, which we quickly prove. 
\begin{proposition}\label{thm:quick_properties}
Let $C \subset \R^d$ be compact. Then \begin{itemize}
    \item[(i)] If $\varepsilon_1 < \varepsilon_2$ then $\tau_{\varepsilon_2}(C) \leq \tau_{\varepsilon_1}(C)$
    \item[(ii)] $\tau_{\varepsilon}(C) \rightarrow \tau(C)$ as $\varepsilon \rightarrow 0$.
\end{itemize}
\end{proposition}
\begin{proof}
 For the first claim, if $(1-\varepsilon_1)\diam(G_n) < \diam(G_i)$ then $$(1-\varepsilon_2)\diam(G_n) < (1-\varepsilon_1)\diam(G_n) < \diam(G_i).$$ So if $i \in H_{\varepsilon_1 , n}$ then $i \in H_{\varepsilon_2 , n}$. Therefore, we include possibly more gaps into consideration with indices in the set $H_{\varepsilon_2, n}$ and thus, the distance to $G_n$ can only shrink. So $\tau_{\varepsilon_2}(C) \leq \tau_{\varepsilon_1}(C).$
 
 Now for the second claim, we can rewrite $\tau(C)$ as is done in Lemma \ref{thm:another_thickness}. Then note, in terms of $\limsup$ and $\liminf$ sets, we have $\lim_{m \rightarrow \infty} H_{1/m, n} = \Lambda_n$ and thus, as $\varepsilon \rightarrow 0$ we get $H_{\varepsilon ,n} \rightarrow \Lambda_n$ which shows $\tau_{\varepsilon}(C) \rightarrow \tau(C)$.
\end{proof}

The next two lemmas will tell us that the thickness of $C$ under $g_{x,t}$ either is not changed at all or only changes slightly when restricting our viewpoint to a small section of $C$. We break up the two lemmas by whether or not $C$ has any bounded gaps.
\begin{lemma}\label{thm:thickness_is_preserved_dimension_d}
Let $C \subset \R^d$ compact with $\tau(C) > 0$ and assume that $C$ has no bounded gaps. Then there exists a rectangle (hyperrectangle in $\R^d$) $R \subset C$ such that $\tau(g_{\widetilde{x},t}(R)) = \tau(C) $ where $\widetilde{x}$ is being taken from a neighborhood with center $x$.
\end{lemma}
\begin{proof}
First note that since $g_{x,t}$ is continuous, the image of any compact set is also compact. So it is valid for us to talk about the thickness of the image of a compact set under $g_{x,t}$. Now because we are assuming $C$ has no bounded gaps, by definition this implies either $\tau(C) = 0$ or $\tau(C) = +\infty$ and since $\tau(C)>0$ we have that $\tau(C) = +\infty$ and that $C$ has non-empty interior. So there exists a ball $B \subset C$ and a (filled-in) rectangle $R \subset B \subset C$. Then recall that each $g^i$ only takes in one coordinate for their argument. So because of this and because each $g^i$ is decreasing and continuous, $g_{x,t}(R)$ is a filled-in rectangle and therefore,  $\tau(g_{x,t}(R)) = + \infty$. The same argument that will be given below in the proof of Lemma \ref{thm:thickness_nearly_preserved_dimension_d} also shows that we can have $\tau(g_{\widetilde{x},t}(R))  = + \infty$ where $\widetilde{x}$ is being taken from a neighborhood of $x$. Thus, in the case when $C$ has no bounded gaps, we simply restrict to this trivial case where thickness is in fact preserved.
\end{proof}

The calculations in the rest of this document will be more convenient if we can work with a specific type of boundary point.
\begin{definition}\label{def:right_endpoint}
Let $C \subset \R^d$ be compact with $H$ begin some gap of $C$. Then we call $u$  a \textbf{right endpoint} of $H$ provided that for any $\varepsilon >0$ we have $u + (\varepsilon, \cdots , \varepsilon) \notin H$.
\end{definition}

\begin{remark}
\normalfont
Intuitively, a right endpoint of a gap $H$ is a boundary point which sits at the "top right" of $H$. Clearly every gap will have a right endpoint.
\end{remark}

Now we state a key lemma which tells us how thickness of $C$ is affected under continuously differentiable mappings where $C$ has at least one bounded gap.  Furthermore, in the next lemma we will take $x=(x_1,\cdots , x_{2d}) \in \R^{2d}$ such that $x_1 = x_i$ for all $i \in \{1, \cdots, d\}$. But this is not a concern for us because by Lemma \ref{thm:nice_conditions}, we can always rotate and translate our sets so that we are in this case. This lemma is a higher dimensional version of Lemma 3.4 in \cite{mcdonald_taylor}. 
\begin{lemma}[Thickness is nearly preserved]\label{thm:thickness_nearly_preserved_dimension_d}
Let $C \subset \R^d$ compact with $\tau(C) > 0$ and let $u$ be a right endpoint of a bounded gap of $C$. Find $x \in \R^{2d}$ such that $x_1 = x_i$ for all $i \in \{1, \cdots ,d\}$ and for which $u$ is in the domain of $g_{x,t}$ and such that $\frac{dg^i}{du_i} \neq 0$ for all $i \in \{1, \cdots , d\}$. Then for every $\varepsilon > 0$, there exists $\delta > 0$ such that $$\tau \left( g_{\widetilde{x},t}(C \cap ([u_1, u_1 + \delta] \times \cdots  \times [u_d, u_d + \delta])) \right) > \tau_{2\varepsilon - \varepsilon^2}(C) (1-\varepsilon)$$ where $\widetilde{x}$ is being taken from a neighborhood with center $x$.
\end{lemma}
\begin{proof}
Let $\varepsilon > 0$ be given. Note that since $x_1 = x_i$ for $i \in \{1, \cdots , d\}$, this means $$\frac{dg^i}{dz_i} = \frac{x_1 - z_i}{\sqrt{\frac{1}{d}t^2 - (x_1-z_i)^2}}. $$ Define $f : \R \rightarrow \R$ by $f(z) = \frac{x_1 - z}{\sqrt{\frac{1}{d}t^2 - (x_1-z)^2}}$. Then $f$ is continuous and for $z \in \R^d$ we have $f(z_i) = d g^i /dz_i$ for $i \in \{1, \cdots ,d\}$. Furthermore, since each $dg^i/du_i$ is nonzero, $f$ will also be nonzero in a neighborhood of $u$. With this and because $f$ is continuous, there exists $\widetilde{\delta} > 0$ such that for any $z \in [u_1, u_1 + \delta] \times \cdots \times [u_d, u_d + \delta]$ where $\delta = \widetilde{\delta}/\sqrt{d}$, we have \begin{equation}\label{eq:general_inequality_1}\left| \frac{|dg^i / dz_i|}{|d g^j / dz_j|} - 1 \right|=  \left| \frac{ |f(z_i)|}{|f(z_j)|} - 1 \right| < \varepsilon \end{equation} for any $i,j \in \{1, \cdots, d\}$ since this means $|z_i - z_j| \leq \sqrt{d} \cdot \delta = \widetilde{\delta}$ which invokes continuity of $f$. We can also do this for a neighborhood of $x$ values. Indeed, take $z \in \left[u_1 + \frac{\delta}{2d}, u_1 + \frac{(2d-1)\delta}{2d}\right] \times \cdots \times \left[u_d + \frac{\delta}{2d}, u_d + \frac{(2d-1)\delta}{2d}\right]$ and $\alpha_i \in \left[ -\frac{\delta}{2d}, \frac{\delta}{2d}\right]$ for $i \in \{1, \cdots , d\}$. Then by setting $\widetilde{x} = (x_1 + \alpha_1, \cdots , x_d + \alpha_d)$ we will have that $$ \frac{\widetilde{x}_i - z_i}{\sqrt{\frac{1}{d}t^2 - (\widetilde{x}_i - z_i)^2}} = \frac{x_1 - (z_i - \alpha_i)}{\sqrt{\frac{1}{d}t^2 - (x_1 - (z_i - \alpha_i))^2}} = \frac{dg^i}{d(z_i - \alpha_i)} $$ and $z- (\alpha_1 , \cdots , \alpha_d) \in [u_1, u_1 + \delta]\times \cdots \times [u_d, u_d + \delta]$. So again if we take $z \in \left[u_1 + \frac{\delta}{2d}, u_1 + \frac{(2d-1)\delta}{2d}\right] \times \cdots \times \left[u_d + \frac{\delta}{2d}, u_d + \frac{(2d-1)\delta}{2d}\right]$ and allowing $\widetilde{x}$ to have the form $(x_1 + \alpha_1, \cdots , x_d + \alpha_d)$, this will still invoke continuity of $f$ so that the ratio $\frac{|dg^i/d(z_i-\alpha_i)|}{|dg^j/d(z_j-\alpha_j)|}$ is still close to $1$ as this is functionally the same as allowing $z \in [u_1, u_1 + \delta] \times \cdots \times [u_d, u_d + \delta]$, but now this allows us to let the $x$-values we are observing to be taken from an entire neighborhood. Explicitly, the neighborhood of $x$-values is the set $x + \left( \left[-\frac{\delta}{2d}, \frac{\delta}{2d} \right] \times \cdots \times \left[-\frac{\delta}{2d}, \frac{\delta}{2d} \right] \right)$ which also shows that the center of this neighborhood is $x$. From here on out, we will abuse notation and use the term $dg^i / dz_i$ when we really mean $dg^i / d(z_i - \alpha_i)$ and we will use $g_{x,t}$ when we really mean $g_{x+\alpha,t}$ for $\alpha_i \in \left[-\frac{\delta}{2d}, \frac{\delta}{2d} \right]$ and $\alpha = (\alpha_1 , \cdots , \alpha_d)$. 

Before going further, we make a quick note about $f$ being well-defined. Since we are assuming $u$ is in the domain of $g_{x,t}$, we have that each $g^i(u_i) = x_{i+d} + \sqrt{\frac{1}{d}t^2 - (x_i - u_i)^2}$ is well-defined. So then clearly we can find an interval $I$ such that for any $\eta_i \in I_i$, the function $g^i(u_i + \eta_i)$ is also well-defined. Thus, $g_{x,t}$ is well-defined for an entire neighborhood, say $N$, for which $u \in N$. This further implies that each $u_i+\eta_i$ is also in the domain of $f$ where $\eta_i \in I_i$. Here, we are assuming without loss of generality that $N$ intersects the set $[u_1, u_1 + \delta] \times \cdots \times [u_d, u_d +\delta]$.
\begin{comment}
For a given $x_1$, the domain of $f$ is going to be  an open interval since $f$ has asymptotes as $(x_1-(z_i - \alpha_i))^2$ approaches $\frac{1}{d}t^2$. Then because we are assuming $u$ is in the domain of $g_{x,t}$
\end{comment}
\begin{comment}
Since we are assuming $u$ is in the domain of $g_{x,t}$ we can take some neighborhood, say $N$, for which $u \cap N \neq \emptyset$ and for which $g_{x,t}$ is also defined, and thus continuous, on $N$. This implies $f$ is defined, coordinate-wise, on $N$. In other words, for $z = (z_1, \cdots , z_d) \in N$, the object $f(z_i)$ is defined for all $i$. Here we are assuming $N$ intersects the set $[u_1, u_1 + \delta] \times \cdots \times [u_d, u_d +\delta]$.
\end{comment}

Resuming, note that any gap in  $A_1 := g_{x,t}(C \cap ([u_1, u_1 + \delta] \times \cdots \times [u_d, u_d + \delta]))$ is the image of a gap in $A_2 := C \cap ([u_1, u_1 + \delta] \times \cdots \times [u_d, u_d + \delta])$. In particular, all gaps in $A_1$ will be of the form $g_{x,t}(G_n)$ where $G_n$ is a gap in $A_2$. After possibly reordering the original gaps, we then produce a new list of gaps of $A_1$, say $\{g_{x,t}(G_n)\}_{n=1}^{\infty}$, with $\diam(g_{x,t}(G_n)) \geq \diam(g_{x,t}(G_{n+1}))$. By Fact \ref{thm:useful_fact}, $\diam(g_{x,t}(G_n)) = \diam(\partial g_{x,t}(G_n))= \diam(g_{x,t}(\partial G_n))$. Thus, for $\varepsilon$-thickness we can take $u_n \in \partial G_n$, for valid $G_n$, and have $$\tau_{\varepsilon}(A_1, g_{x,t}(u_n)) = \frac{\dist(g_{x,t}(u_n), \bigcup_{i \in H_{\varepsilon,n}} g_{x,t}(G_i) \cup E)}{\diam(g_{x,t}(G_n))}$$ where $E$ now denotes the unbounded gap of $A_1$ and $H_{\varepsilon,n} = \{i \neq n : (1-\varepsilon) \diam(g_{x,t}(G_n)) < \diam(g_{x,t}(G_i)) \}$. Now let $v_n \in \partial G_n$ such that $g_{x,t}(G_n)$ is a gap in $A_1$. Note that since we are assuming that we have thick compact sets, then the thickness of neither is equal to zero. So we always get nonzero values in the numerator for the expression of $\tau_{\varepsilon}(A_1, g(v_n))$. So we can find some element, say $v_j \in \partial G_j$ where $G_j$ could also be the unbounded gap of $A_1$ such that $(1-\varepsilon)\diam(g_{x,t}(G_n)) < \diam(g_{x,t}(G_j))$ and $\dist(g_{x,t}(v_n), \bigcup_{i \in H_{\varepsilon,n}} g_{x,t}(G_i) \cup E) = \| g_{x,t}(v_n) - g_{x,t}(v_j) \|$. Set $A_3:=[u_1, u_1 + \delta] \times \cdots \times [u_d, u_d + \delta]$. By Lemma \ref{thm:MVT}, there exists $z \in A_3$ such that $\| g_{x,t}(v_n) - g_{x,t}(v_j)\| = \| g_{x,t}^{'}(z) (v_n - v_j) \|$. Furthermore, $\diam(g_{x,t}(G_n)) = \| g_{x,t}(a_n) - g_{x,t}(b_n)\|$ for some $a_n, b_n \in \partial G_n$. Again by Lemma \ref{thm:MVT} we have that there exists some $w \in A_3$ such that $\| g_{x,t}(a_n) - g_{x,t}(b_n)\| = \| g_{x,t}^{'}(w)(a_n - b_n)\|.$ This gives \begin{align*}
\tau_{\varepsilon}(A_1, g_{x,t}(v_n)) &= \frac{\| g_{x,t}^{'}(z) (v_n - v_j)\|}{\diam(g_{x,t}(G_n))} \\ &= \frac{\|g_{x,t}^{'}(z) (v_n - v_j)\|}{\|g_{x,t}^{'}(w) (a_n - b_n)\|} \\ &\geq \frac{\| g_{x,t}^{'}(z) (v_n - v_j)\|}{ \|g_{x,t}^{'}(w)\|_{op}\| a_n - b_n \|}.
\end{align*}
To get a further lower estimate on this quantity, we can assume without loss of generality that $\left| \frac{dg^2}{dz_2} \right| = \min \left\{ \left| \frac{dg^1}{dz_1} \right|, \cdots , \left| \frac{dg^d}{dz_d} \right| \right\} $. Then \begin{align*}
\| g_{x,t}^{'}(z)(v_n - v_j) \| &= 
\left\| 
\begin{bmatrix}
\frac{dg^1}{dz_1} & 0 & \cdots & 0 \\ \vdots & \ddots & & \vdots \\
0 & 0 & \cdots & \frac{dg^d}{dz_d}
\end{bmatrix}
\begin{bmatrix}
v_{n,1} - v_{j,1} \\ \vdots \\
v_{n,d} - v_{j,d}
\end{bmatrix}
\right\| \\ &= \sqrt{ \left( \frac{dg^1}{dz_1}\right)^2 \left( v_{n,1} - v_{j,1} \right)^2 + \cdots + \left( \frac{dg^d}{dz_d}\right)^2 \left( v_{n,d} - v_{j,d} \right)^2} \\ &\geq \sqrt{ \left( \frac{dg^2}{dz_2}\right)^2 \left( v_{n,1} - v_{j,1} \right)^2 + \cdots + \left( \frac{dg^2}{dz_2}\right)^2 \left( v_{n,d} - v_{j,d} \right)^2} \\ &= \left| \frac{dg^2}{dz_2} \right| \| v_n - v_j\|.
\end{align*}
Then \begin{align*}
\tau_{\varepsilon}(A_1, g_{x,t}(v_n)) &\geq \frac{ \left| \frac{dg^2}{dz_2} \right| \| v_n - v_j\|}{\| g_{x,t}^{'}(w)\|_{op} \| a_n - b_n \|} \\ &\geq \frac{ \left| \frac{dg^2}{dz_2} \right| \| v_n - v_j\|}{\| g_{x,t}^{'}(w)\|_{op} \diam(G_n)}.
\end{align*}
Using Equation \eqref{eq:op_norm}, without loss of generality we can assume $\| g_{x,t}^{'}(w)\|_{op} = \left| \frac{dg^1}{dw_1}\right|.$ By Equation \eqref{eq:general_inequality_1} we have $$\frac{\left| \frac{dg^2}{dz_2} \right|}{\left| \frac{dg^1}{dw_1} \right|} > 1 - \varepsilon$$ which yields \begin{align*}
\tau_{\varepsilon}(A_1, g_{x,t}(v_n)) &\geq \frac{ \left| \frac{dg^2}{dz_2} \right| \| v_n - v_j\|}{\left| \frac{dg^1}{dw_1}\right| \diam(G_n)} \\ &> (1-\varepsilon) \frac{ \| v_n - v_j\|}{\diam(G_n)}.
\end{align*}
Note that $v_n$ and $v_j$ were chosen such that $\| g_{x,t}(v_n) - g_{x,t}(v_j)\| = \dist\left( g_{x,t}(v_n) , \bigcup_{i \in H_{\varepsilon,n}} g_{x,t}(G_i) \cup g_{x,t}(E) \right)$ and recall that since $j \in H_{\varepsilon,n}$ we have $(1-\varepsilon) \diam (g_{x,t}(G_n)) < \diam(g_{x,t}(G_j))$. Note that $j$ being in $H_{\varepsilon,n}$ refers to the ordering that was placed on $\{g_{x,t}(G_n)\}_{n=1}^{\infty}$, not necessarily the ordering that we placed on $\{G_n\}_{n=1}^{\infty}$. We will now try to find $\eta > 0$ such that $j \in H_{\eta, n}$ where this now refers to the ordering that we placed on $\{G_n\}_{n=1}^{\infty}$ which will provide a lower estimate on $\| v_n - v_j\|$. Let $\alpha^1 ,\beta^1 \in \partial G_n$ such that $\diam(G_n) = \| \alpha^1 - \beta^1 \|$ and $\alpha^2 , \beta^2 \in \partial G_j$ such that $\diam(g_{x,t}(G_j)) = \| g_{x,t}(\alpha^2) - g_{x,t}(\beta^2) \|$. Then $$(1-\varepsilon) \| g_{x,t}(\alpha^1) - g_{x,t}(\beta^1) \| \leq (1-\varepsilon) \| g_{x,t}(a_n) - g_{x,t}(b_n) \| <  \| g_{x,t}(\alpha^2) - g_{x,t}(\beta^2) \| .$$ As before, we can find $\rho^1, \rho^2 \in A_3$ such that $\| g_{x,t}(\alpha^1) - g_{x,t}(\beta^1) \| = \| g_{x,t}^{'}(\rho^1)(\alpha^1 - \beta^1)\|$ and $\| g_{x,t}(\alpha^2) - g_{x,t}(\beta^2) \| = \| g_{x,t}^{'}(\rho^2)(\alpha^2 - \beta^2)\|$. By a similar calculations as above, we can assume without loss of generality that $$\| g_{x,t}^{'}(\rho^1)(\alpha^1 - \beta^1)\| \geq \left| \frac{dg^1}{d\rho_1^1}\right| \|\alpha^1 - \beta^1\|$$ and $$\| g_{x,t}^{'}(\rho^2)(\alpha^2 - \beta^2)\| \leq \| g_{x,t}^{'} (\rho^2)\|_{op} \|\alpha^2 - \beta^2\| = \left| \frac{dg^2}{d\rho_2^2} \right| \|\alpha^2 - \beta^2\|.$$ Along with Equation \eqref{eq:general_inequality_1} this implies \begin{align*}
\diam(G_j) &\geq \| \alpha^2 - \beta^2\| \\ &> (1-\varepsilon) \frac{ \left| \frac{dg^1}{d\rho_1^1} \right|}{\left| \frac{dg^2}{d\rho_2^2} \right|} \| \alpha^1 - \beta^1\| \\ &> (1-\varepsilon)(1-\varepsilon) \| \alpha^1 - \beta^1\| \\ &= (1-(2\varepsilon - \varepsilon^2)) \diam(G_n).
\end{align*}
Therefore, $j \in H_{2\varepsilon - \varepsilon^2, n}$ where this refers to the ordering placed on $\{ G_n\}_{n=1}^{\infty}$. So \begin{align*}
\tau_{\varepsilon}(A_1, g_{x,t}(v_n)) &\geq (1-\varepsilon) \frac{ \| v_n - v_j\|}{\diam(G_n)} \\ &\geq (1-\varepsilon) \frac{ \dist\left( v_n, \bigcup_{i \in H_{2\varepsilon - \varepsilon^2, n}} G_i \cup E \right)}{\diam(G_n)} \\ &= (1-\varepsilon) \tau_{2\varepsilon - \varepsilon^2}(C , v_n).
\end{align*}
Taking infimums produces $$ \tau(A_1) \geq \tau_{\varepsilon} (A_1) > \tau_{2\varepsilon - \varepsilon^2} (C) (1-\varepsilon)$$ which finishes the proof.
\end{proof}
This lemma ensures that we can find a lower bound on $\tau(g_{x,t}(C))$ in terms of $C$. Thus, even though it can be difficult to see how $C$, globally, is affected by $g_{x,t}$, we can restrict $C$ in a local manner and use the above lemma to ensure that the thickness of $C$ does not change much under $g_{x,t}$.

In the proof of Theorem \ref{thm:general}, we will require our pinned points to not share a coordinate and Lemma \ref{thm:different_coordinates} ensures that we can do this. The proof relies on the fact that unions of hyperplanes have zero thickness, which we now present.

\begin{fact}\label{thm:hyperplanes_zero_thickness}
Let $C \subset \R^d$ be compact such that $C \subseteq \bigcup_{i=1}^n J_i$ where each $J_i$ is a bounded hyperplane. Then $\tau(C) = 0$.
\end{fact}
\begin{proof}
Let $C_1$ denote the part of $C$ that is a subset of $J_1$, i.e. $C_1 = C \cap J_1$. Then since $C_1$ can only be removing bounded gaps of $C$ we will have $\tau(C) \leq \tau(C_1)$. Furthermore, the only path connected component of $C_1^C$ is its unbounded gap and because hyperplanes have empty interior, this implies $C_1$ has empty interior as well. By the definition of thickness, this yields $\tau(C_1) = 0$ showing the claim.
\end{proof}

\begin{lemma}\label{thm:different_coordinates}
Assume $ \tau(C_1), \tau(C_2) > 0$ where $C_1, C_2 \subset \R^d$. Then given any $x^1 \in C_1 \times C_2$ we can always find $x^2, \cdots , x^{k+1} \in C_1 \times C_2$ such that every element in the set $\{x_i^1, x_i^2, \cdots , x_i^{n}\}$ does not share a coordinate for all $ i \in \{1, \cdots , 2d \}$. Thus, in particular we have $x_i^j \neq x_i^k$ for all $i \in \{1, \cdots , 2d\}$ and $j,k \in \{ 1, \cdots , n\}$, i.e. none of $x^1, \cdots , x^{k+1}$ share a coordiante.
\end{lemma}
\begin{proof}
We prove this by induction. Since we are always given $x^1$, our induction starts at $n=2$. Thus, we will find $x^2 \in C_1 \times C_2$ so that each $x_i^1 \neq x_i^2$ for $i \in \{1, \cdots , 2d\}$. Assume for contradiction that we cannot do this. Define $$H_i := \{ (x_1^1 + \alpha_1, \cdots , x_{i-1}^1 + \alpha_{i-1}, x_i^1, x_{i+1}^1 + \alpha_{i+1}, \cdots , x_{2d}^1 + \alpha_{2d}): - \diam(C_1 \times C_2) \leq \alpha_j \leq \diam(C_1 \times C_2 \}$$ for all $i \in \{1, \cdots , 2d\}$, i.e. $H_i$ is a bounded hyperplane with one of the coordinates fixed. Note that our contradiction assumption means that for all $x \in C_1 \times C_2$, we have that $x$ and $x^1$ must share some coordinate. Therefore, $x \in H_i$ for some $i$ which further implies that $C_1 \times C_2 \subseteq \bigcup_{1 \leq k \leq 2d} H_k$. But if there exists $\alpha_1, \cdots, \alpha_{2d} \in \R \backslash \{ 0\}$ such that $(x_1^1 + \alpha_1 , \cdots , x_d^1 + \alpha_d) \in C_1$ and $(x_{d+1}^1 + \alpha_{d+1}, \cdots x_{2d}^1 + \alpha_{2d}) \in C_2$ then $(x_1^1 + \alpha_1, \cdots , x_{2d}^1 + \alpha_{2d}) \in C_1 \times C_2$ and $(x_1^1 + \alpha_1, \cdots , x_{2d}^1 + \alpha_{2d}) \notin \bigcup_{1 \leq k \leq 2d} H_k$ which would give us the desired contradiction. Clearly we can do this. Indeed, if we couldn't then by defining $$J_i := \{ (x_1^1 + \beta_1, \cdots , x_{i-1}^1+\beta_{i-1}, x_i^1, x_{i+1}^{1} \beta_{i+1}, \cdots , x_d^1 +\beta_d): -\diam(C_1) \leq \beta_j \leq \diam(C_1) \}  $$ we would have $C_1 \subseteq \bigcup_{1 \leq k \leq d} J_k$ and that each $J_k$ is a bounded hyperplane. So by Fact \ref{thm:hyperplanes_zero_thickness} this means $\tau(C_1) = 0$ which gives us a contradiction. We similarly come to the same conclusion for $C_2$. This proves the base case. 

Now assume the claim is true for all $n \leq k$ and consider the $k+1$ case. Again assume for contradiction that we cannot find $x^{k+1} \in C_1 \times C_2$ such that $x_i^{k+1} \neq x_i^j$ for $i \in \{1, \cdots, 2d\}$ and $j\in \{1, \cdots , k \}$. We proceed similarly as in the base case. Since the points $x^2, \cdots, x^k$ are generated in terms of $x^1$, we can write $x^j = (x_1^1 + \eta_1^j, \cdots , x_{2d}^j + \eta_{2d}^j)$ for $-\diam (C_1 \times C_2) \leq \eta_i^j \leq \diam(C_1 \times C_2).$ Then define \begin{align*} H_i^j := \left\{ \begin{aligned}
(x_1^1 + \eta_1^j + \alpha_1^j &, \cdots , x_{i-1}^1 + \eta_{i-1}^j + \alpha_{i-1}^j, x_i^1 + \eta_i, x_{i+1}^1 + \eta_{i+1}^j + \alpha_{i+1}^j, \cdots ,x_{2d}^1 + \eta_{2d}^j + \alpha_{2d}^j \\ &: -\diam (C_1 \times C_2) \leq \alpha_i^j \leq \diam(C_1 \times C_2) \end{aligned}
\right\}. \end{align*} Again by our contradiction assumption we have $C_1 \times C_2 \subseteq \bigcup_{1 \leq j \leq k} \bigcup_{1 \leq i \leq 2d} H_i^j$. Similarly, if there exists $\alpha_1, \cdots , \alpha_{2d} \in \R \backslash \{0 \}$ and some $j$ for which $(x_1^1 + \eta_1^j + \alpha_1, \cdots , x_d^1 + \eta_d^j + \alpha_d) \in C_1$ and $(x_{d+1}^1 +\eta_{d+1}^j + \alpha_{d+1}, \cdots, x_{2d}^1 + \eta_{2d}^j + \alpha_{2d}) \in C_2$ then $(x_1^1 + \eta_1^j + \alpha_1, \cdots , x_{2d}^1 + \eta_{2d}^j + \alpha_{2d}) \in C_1 \times C_2$ and $(x_1^1 + \eta_1^j + \alpha_1, \cdots , x_{2d}^1 + \eta_{2d}^j + \alpha_{2d}) \notin \bigcup_{1 \leq j \leq k} \bigcup_{1 \leq i \leq 2d} H_i^j$ which would give us the desired conclusion. If we couldn't do this, then for all $j$, by defining 
\begin{align*} J_i^j := \left\{ \begin{aligned}
(x_1^1 + \eta_1^j + \beta_1^j &, \cdots , x_{i-1}^1 + \eta_{i-1}^j + \beta_{i-1}^j, x_i^1 + \eta_i, x_{i+1}^1 + \eta_{i+1}^j + \beta_{i+1}^j, \cdots ,x_{2d}^1 + \eta_{2d}^j + \beta_{2d}^j \\ &: -\diam (C_1 \times C_2) \leq \beta_i^j \leq \diam(C_1 \times C_2) \end{aligned}
\right\} \end{align*} we would have $C_1 \subseteq \bigcup_{1 \leq j \leq k} \bigcup_{1 \leq i \leq d} J_i^j$ and by the same argument as before, this would $\tau(C_1) = 0$, a contradiction. Then the same argument works for $C_2$ which gives us the desired conclusion.
\end{proof}
This lemma allows us to generate a set of points which will be useful in the proof of Theorem \ref{thm:general}. The next lemma helps to give a criterion that will describe the convex hull of a small subset of a compact set. In particular, this criterion will be used in the proof of Theorem \ref{thm:general}.

\begin{lemma}\label{thm:find_delta_general}
Let $C \subset \R^d$ compact such that $\tau(C) > 0$. Let $u \in C$ be a right endpoint of some gap of $C$. Then we can find $\delta = (\delta_1, \cdots, \delta_d)$ such that $u+\delta \in C$ where $\delta_i >0$. Furthermore, for any $\eta > 0$ we can, possibly, shrink $\delta$ so that $|\delta| < | \eta|$ and we can have $\delta_1 = \delta_i$ for all $i \in \{1, \cdots ,d\}$.
\end{lemma}
\begin{proof}
Assume for contradiction that we cannot find $\delta$ such that $u+\delta \in C$. So now let $\eta > 0$ be given and restrict $\delta$ such that $|\delta| < |\eta|$. Let $G_n$ denote the gap for which $u \in \partial G_n$. Define $$L_u := \{u(1-t) + t(u+\delta): t\in (0,1) \},$$ i.e. $L_u$ is the line segment connecting $u$ and $u+\delta$ but not including $u$ and $u+\delta$. Recall that since $u$ is a right endpoint of $G_n$, we have that for any $\varepsilon > 0$, $u + (\varepsilon, \cdots , \varepsilon) \notin G_n$, so we can take $\delta$ such that $\delta_1 = \delta_i$ for all $i \in \{1, \cdots d\}$. Furthermore, by the contradiction assumption we have $L_u \cap C = \emptyset$ and therefore there exists some index set $\Lambda$ such that $L_u \subset \bigcup_{i \in \Lambda}G_i$ where the $G_i$ are gaps of $C$ and such that $L_u \cap G_i \neq \emptyset$. We claim there exists a single gap of $C$, say $G_i$, such that $L_u \subset G_i$. To see this, assume this is not true. Thus, we have that  $L_u \cap G_i \neq \emptyset$ and $G_i \cap G_j = \emptyset$ for all $i,j \in \Lambda$. This implies that $L_u$ is a disconnected space which is a contradiction. 
\begin{comment}Then there exists at least two distinct gaps of $C$, say $H_1$ and $H_2$, such that $L_u \cap H_1 \neq \emptyset$ and $L_u \cap H_2 \neq \emptyset$. Also, recall that the gaps we use the construct $C$ are pairwise disjoint and therefore, $(H_1 \cap L_u) \cap ( H_2 \cap L_u) = \emptyset$. Thus, $H_1 \cap L_u \subset (H_1 \cap L_u) \cup (H_2 \cap L_u)$ and $H_2 \cap L_u \subset (H_1 \cap L_u) \cup (H_2 \cap L_u)$ and taking unions yields $$(H_1 \cap L_u) \cup (H_2 \cap L_u) \subset  (H_1 \cap L_u) \cup (H_2 \cap L_u)$$ which is a contradiction.
\end{comment}
Therefore, $L_u \subset G_i$ for some gap $G_i$ of $C$. But this implies $\dist(G_i , G_n) = 0$. Without loss of generality, assume $\diam(G_i) \leq \diam(G_n)$. Then $$0 < \tau(C) \leq \frac{\dist\left(G_i , \bigcup_{j \in \Lambda_i} G_j \cup E\right)}{\diam(G_i)} \leq \frac{\dist(G_i, G_n)}{\diam(G_i)} = 0 $$ a contradiction.  Note that $\Lambda_i$ is being defined as in Lemma \ref{thm:another_thickness}.
\end{proof}

Checking that a compact set is not contained in a gap of another compact set is tricky (which is a hypothesis of Theorem \ref{thm:gap_lemma}), so the next lemma and theorem gives a sufficient and easier criterion to work with. Let $\conv(\cdot)$ denote the convex hull of a set.

\begin{lemma}\label{thm:non-empty}
Let $C \subset \R^d$ be a compact set such that $\tau(C) > 0$. Let $U = \conv(C)^{\circ}$. Then $U$ is non-empty.
\end{lemma}
\begin{proof}
 Assume for contradiction that $U = \emptyset$. Then note that $\conv(C)$ has no bounded gaps. Then since $U = \emptyset$, by definition of thickness, this implies $\tau(\conv(C)) = 0$. But note that $\tau(C) \leq \tau(\conv(C))$ since $\conv(C)$ has simply removed all of the bounded gaps of $C$ which then means $$0 < \tau(C) \leq \tau(\conv(C)) = 0 $$ a contradiction.
\end{proof}

\begin{theorem}\label{thm:convex_hull}
Let $C_1$ and $C_2$ be compact sets in $\R^d$ with $\tau(C_1), \tau(C_2) > 0$ and such that their convex hulls are linked. That is, by setting $U_1 =\conv(C_1)^{\circ}$ and $U_2 = \conv(C_2)^{\circ}$, we have that $U \cap V \neq \emptyset$, $\left( \partial U\right) \backslash V \neq \emptyset$, and $\left( \partial V \right) \backslash U \neq \emptyset$. Then neither $C_1$ or $C_2$ is contained in a gap of the other.
\end{theorem}
\begin{proof}
Set $U_1 = \conv(C_1)^{\circ}$ and $U_2 = \conv(C_2)^{\circ}$. By Lemma \ref{thm:non-empty} both $U_1$ and $U_2$ are non-empty. Assume for contradiction that $C_1$ is contained in a gap of $C_2$, say $G^2$. So $C_1 \subseteq G^2$. Note that $C_1 \subset G^2$ since $G^2$ is open and $C_1$ is closed. Then $\conv(C_1) \subseteq \conv(G^2)$. Note that since $G^2$ is open, $\conv(G^2)$ is open. Thus, $$U_3 := \conv(G^2)^{\circ} = \conv(G^2).$$ Note that $\partial U_1 \subseteq \partial \conv(C_1)$. Then if $\partial \conv(C_1) \subset U_3$ we're done since this would imply $$\emptyset \neq (\partial U_1) \backslash U_2 \subseteq \left( \partial \conv(C_1) \right) \backslash U_2 \subseteq  \left( \partial \conv(C_1) \right) \backslash U_3 = \emptyset$$ which yields a contradiction. So now assume $\partial \conv(C_1) \supseteq U_3$. To see that this implies these sets are equal, take $x \in \partial \conv(C_1)$. Then since $C_1$ is closed, $\conv(C_1)$ is also closed and thus we have $\partial \conv(C_1) \subseteq \conv(C_1)$ and therefore, $x \in \conv(C_1)$. But $$x \in \conv(C_1) \subseteq \conv(G^2) = \conv(G^2)^{\circ} = U_3$$ and therefore, $U_3 = \partial \conv(C_1)$ which is a contradiction since $U_3$ is open and $\partial \conv(C_1)$ is closed. This gives the desired conclusion.
\end{proof}

Thus, rather than checking that the two sets are not contained in a gap of the other, we will use the above criterion. We can now finally prove the main ingredient needed which is Theorem \ref{thm:general}. 

\begin{proof}[Proof of Theorem \ref{thm:general}]
 Let $u^1$ and $u^2$ be right endpoints of some gap (not necessarily bounded gaps) of $C_1$ and $C_2$ respectively. By Lemma \ref{thm:different_coordinates} we may assume  $(u^1, u^2)$ lives to the upper right of $x^0$, i.e. that for $i \in \{1, \cdots , d\}$ and $j \in \{ d+1, \cdots , 2d\}$ we have $x_i^0 < u_i^1$ and $x_j^0 < u_j^2$. We assume this so that we are in the setting of Lemma \ref{thm:thickness_nearly_preserved_dimension_d}. By Lemma \ref{thm:find_delta_general} we find $\delta^j \in \R^d$ such that $u^j, u^j + \delta^j \in C_j$ for which $\delta_1^j = \delta_i^j$ for all $i$. So set $\widetilde{C_j} := C_j \cap \left( [u_1^i, u_1^i + \delta_1^i] \times \cdots \times [u_d^i , u_d^i + \delta_1^i] \right)$. In particular we can also choose $\delta^i$ small enough so that $\widetilde{C_1}$ is in the domain of $g_{x,t}$ and for which $\tau(\widetilde{C_2}) \tau(g_{x,t}(\widetilde{C_1})) > 1$ by Lemma \ref{thm:find_delta_general} and Lemma \ref{thm:thickness_nearly_preserved_dimension_d}. Note that we are invoking Lemma \ref{thm:thickness_nearly_preserved_dimension_d} for a neighborhood of $x$ values centered at $x^0$. Recall that by Lemma \ref{thm:nice_conditions}, we can translate and rotate our set so that we may also assume $u^1$ and $u^2$ lie on the diagonal, i.e. $u_i^1 = u_1^1$ and $u_i^2 = u_1^2$ for all $i \in \{1, \cdots d\}$, and that $x^0 = \vec{0}$. To use the Gap Lemma we find $(x,t)$ such that $\widetilde{C_2}$ and $g_{x,t}(\widetilde{C_1})$ are linked where $x$ is coming from a neighborhood around $x^0$. The parameters $x$ and $t$ are important for this proof so let $g_{x,t}^i$ denote $g^i$. Note that each $g_{x,t}^i$ is decreasing in its argument. 
\begin{figure}[]
\centering
\begin{tikzpicture}
 
    \node[rectangle,draw, minimum width = 5cm, 
    minimum height = 6cm] (r) at (0,0) {};
    \node[rectangle,draw, minimum width = 5cm, 
    minimum height = 6cm] (r) at (-2,-2) {};
    
    \fill [color=black] (-2.5,3) circle (0.1);
    \node at (-2.5,3.5) {$(u_1^2, u_2^2 + \delta_2^2)$};
    \fill [color=black] (2.5,3) circle (0.1);
    \node at (2.5, 3.5) {$(u_1^2 + \delta_1^2, u_2^2 + \delta_2^2)$};
    \fill [color=black] (2.5,-3) circle (0.1);
    \node at (2.5, -3.5) {$(u_1^2 + \delta_1^2, u_2^2)$};
    \fill [color=black] (-2.5,-3) circle (0.1);
    \node at (-2.5, -3.5) {$(u_1^2, u_2^2)$};
    
    \fill [color=black] (-4.5,1) circle (0.1);
    \node at (-4.5,1.5) {$(g_{x,t}^1(u_1^1 + \delta_1^1), g_{x,t}^2(u_2^1))$};
    \fill [color=black] (0.5,1) circle (0.1);
    \node at (0.5,1.5) {$(g_{x,t}^1(u_1^1), g_{x,t}^2(u_2^1))$};
    \fill [color=black] (-4.5,-5) circle (0.1);
    \node at (-4.5,-5.5) {$(g_{x,t}^1(u_1^1 + \delta_1^1), g_{x,t}^2(u_2^1+\delta_2^1))$};
    \fill [color=black] (0.5,-5) circle (0.1);
    \node at (0.5,-5.5) {$(g_{x,t}^1(u_1^1), g_{x,t}^2(u_2^1+\delta_2^1))$};

\end{tikzpicture}
\caption{How we want the conditions of $U$ to look like in $\R^2$. The (filled in) rectangles represent the convex hulls of $\widetilde{C_2}$ and $g_{x,t}(\widetilde{C_1})$ and thus, their convex hulls will be linked provided that they are in this type of configuration.}
\end{figure}
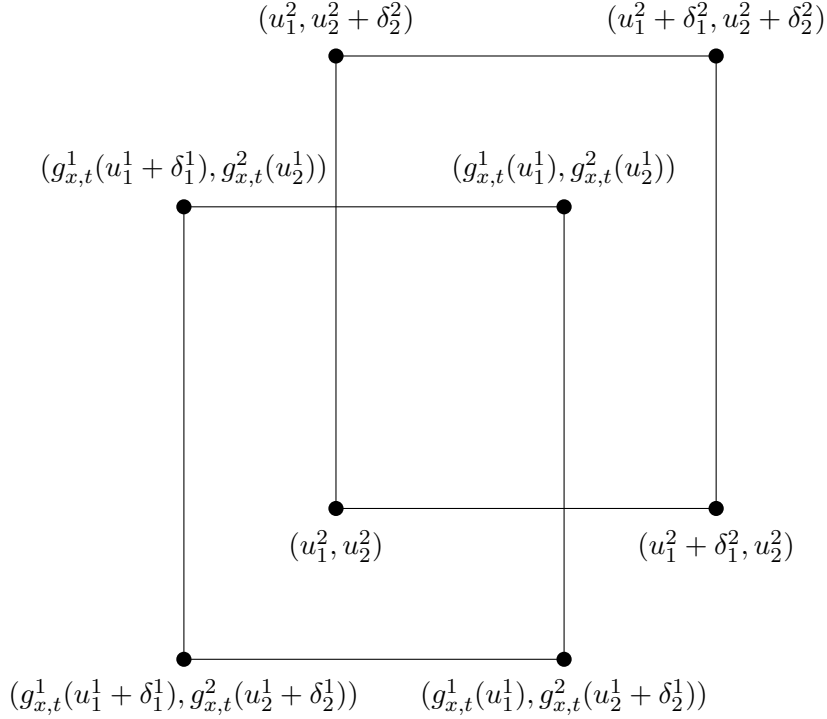
So consider the set $$U:= \{ (x,t) \in \R^{2d} \times \R : g_{x,t}^1(u_1^1 + \delta_1^1) < u_1^2 < g_{x,t}^1(u_1^1) < u_1^2 + \delta_1^2, \cdots , g_{x,t}^d(u_d^1 + \delta_d^1) < u_d^2 < g_{x,t}^d(u_d^1) < u_d^2 + \delta_d^2\}.$$

Again, note that $\delta_1^j = \delta_i^j$. We simply use this definition of $U$ in terms of $\delta_i^j$ to make $U$ look symmetric. Then if $(x,t) \in U$, we have that $\widetilde{C_2}$ and $g_{x , t}(\widetilde{C_1})$ are linked because their convex hulls will be rectangles in $\R^d$ and for $(x,t) \in U$, this guarantees that these rectangles intersect each other and exist outside of each other. Then by Theorem \ref{thm:convex_hull} and Theorem \ref{thm:gap_lemma}, we will get that $\widetilde{C_2} \cap g_{x,t}(\widetilde{C_1}) \neq \emptyset$. Note this means we are using Lemma \ref{thm:find_delta_general}.

We will show that $U$ is a non-empty open set containing a point of the form $(x^0, t) \in U$. The claim then follows because we can take open neighborhoods $S,T$ of $x^0 ,t$ respectively such that $$T \subset \bigcap_{x \in S} \Delta_x(C_1 \times C_2) $$ where $S$ is some subset of the neighborhood found in Lemma \ref{thm:thickness_nearly_preserved_dimension_d}. To see that $U$ is non-empty, set $\frac{1}{d} (t_1)^2 = (x_1^0 - u_1^1)^2 + (x_{d+1}^0 - u_1^2)^2$. By construction, we have $g_{x^0,t_1}^1(u_1^1) = u_1^2$. Since each $g_{x,t}^i$ is increasing in $t$,  for any $t > t_1$ we will have $g_{x^0,t}^1(u_1^1) > u_1^2$. Also, $g_{x,t}(z)$ is a continuous function of $t$ so that when $t$ is sufficiently close to $t_1$ we will get $g_{x^0,t}^1(u_1^1 + \delta_1^1) < u_1^2$ and $g_{x^0,t}^1(u_1^1) < u_1^2 + \delta_1^2$. This satisfies the first condition of $U$.

For the rest of the conditions of $U$, set $\frac{1}{d} (t_i)^2 = (x_i^0 - u_i^1)^2 + (x_{i+d}^0 - u_i^2)^2$. But by assumption, since $x^0 = \vec{0}$, $u_i^1 = u_1^1$, and $u_i^2 = u_1^2$ this yields $$\frac{1}{d} (t_i)^2 = (x_i^0 - u_i^1)^2 + (x_{i+d}^0 - u_i^2)^2 = (u_1^1)^2 + (u_1^2)^2 = (x_1^0 - u_1^1)^2 + (x_{d+1}^0 - u_1^2)^2 = \frac{1}{d} (t_1)^2.$$ So by the same argument as above, the same $t$ which worked for $g_{x^0,t}^1$ will work for $g_{x^0,t}^i$ as well. This satisfies the remaining conditions of $U$ and therefore there exists a point of the form $(x^0, t) \in U$. To see that $U$ is open, note that for a fixed input $z$, each $g_{x,t}^i(z)$ is continuous in $(x,t)$ and therefore, $U$ is open. This satisfies the claim.
\end{proof}

Finally, we conclude this paper with the proof of Theorem \ref{thm:pinned_tree_set} which we show by induction.

\begin{proof}[Proof of Theorem \ref{thm:pinned_tree_set}]
To start the induction, we first let $k=1$ and let $x^0 \in C_1 \times C_2$ denote our fixed point. Then observe $$\Delta_{T_{x}^1} (C_1 \times C_2) = \Delta_{x}(C_1 \times C_2) \backslash \{0\}. $$ Thus, by Theorem \ref{thm:general} there exists and neighborhood $S$ about $x$ such that the set $$\bigcap_{\widetilde{x} \in S} \Delta_{\widetilde{x}}(C_1 \times C_2) \backslash \{0\} \subset \Delta_{x}(C_1 \times C_2) \backslash \{0\} $$ has non-empty interior. This proves the base case. 

Now take $k > 1$. Then define $\Delta_{\widetilde{T_{x}^k}}(E)$ to be the pinned tree distance set of $k+1$ vertices where $x$ is labeled as the first leaf, the first coordinate is always a distance from the fixed point $x$ to the unique vertex which forms an edge with $x$, and the rest of the coordinates behave in the same way that the set $\Delta_{T^{k-1}}(E)$ behaves. In set form, $$\Delta_{\widetilde{T_{x}^k}}(E) = \left\{(|x-x^n|, |x^i-x^j|)_{i \sim j} :
\begin{aligned}
&x^2, \cdots, x^{k+1} \in E, x^i \neq x^j, x \text{ is a leaf}, \\
&x^n \text{ unique vertex forming edge with } x 
\end{aligned}
\right\} $$ where $(|x-x^n|, |x^i-x^j|)_{i \sim j} \in \R^k$. Note that $$\Delta_{T_x^k}(C_1 \times C_2) \supset \Delta_{\widetilde{T_x^k}}(C_1 \times C_2) = \left(\Delta_x(C_1 \times C_2) \backslash \{0\} \right) \times \Delta_{T^{k-1}}(C_1 \times  C_2) .$$ We have already seen that $\Delta_x(C_1 \times C_2) \backslash \{0\}$ has non-empty interior. Furthermore, $$\Delta_{T^{k-1}}(C_1 \times C_2) = \bigcup_{\widetilde{x} \in C_1\times C_2} \Delta_{T_{\widetilde{x}}^{k-1}} (C_1 \times C_2) $$ and by the inductive hypothesis, $\Delta_{T_{\widetilde{x}}^{k-1}}(C_1 \times C_2)$ has non-empty interior for any $\widetilde{x} \in C_1 \times C_2$ which implies $\Delta_{T^{k-1}}(C_1 \times C_2)$ has  non-empty interior, and thus, the claim is satisfied.
\end{proof}

\bibliographystyle{plain}
\bibliography{references}

\end{document}